\documentclass[letterpaper, 11pt]{article}
\usepackage{amsmath,amssymb, amsthm, dsfont}

\usepackage[left=1in,top=1in,right=1in,bottom=1in]{geometry}

\newtheorem{thm}{Theorem}[section]
\newtheorem{lemma}[thm]{Lemma}
\newtheorem*{PM*}{Poincar\'e-Miranda Theorem}
\newtheorem{prop}[thm]{Proposition}

\newtheorem{cor}[thm]{Corollary}

\newtheorem{?}[thm]{Question}
\newtheorem{alg}[thm]{Algorithm}

\theoremstyle{definition}
 \newtheorem{definition}[thm]{Definition}
  \newtheorem{example}[thm]{Example}
 \newtheorem{remark}[thm]{Remark}

\DeclareMathOperator{\In}{in}
\DeclareMathOperator{\argmin}{argmin}
\DeclareMathOperator{\argmax}{argmax}
\DeclareMathOperator{\Trop}{Trop}

\DeclareMathOperator{\LL}{\mathcal{L}}
\DeclareMathOperator{\QM}{QM}

\DeclareMathOperator{\Reg}{Reg}

\newcommand{\R}{\mathbb{R}}

\newcommand{\N}{\mathbb{N}}
\newcommand{\Z}{\mathbb{Z}}
\newcommand{\proj}{\mathbb{P}}
\newcommand{\C}{\mathbb{C}}
\newcommand{\V}{\mathcal{V}}
\newcommand{\A}{\mathcal{A}}
\newcommand{\f}{\overline{f}}
\newcommand{\I}{\overline{I}}

\title{Real radical initial ideals}
\author{Cynthia Vinzant \footnote{University of California, Berkeley, Department of Mathematics, cvinzant@math.berkeley.edu}}

\date{\today}
\begin{document}

\maketitle
\begin{abstract}
We explore the consequences of an ideal $I \subset \R[x_1, \hdots, x_n]$ having a real radical initial ideal, both for the geometry of the real variety of $I$ and as an application to sums of squares representations of polynomials. We show that if $\In_w(I)$ is real radical for a vector $w$ in the tropical variety, then $w$ is in the logarithmic set of the real variety. We also give algebraic sufficient conditions for $w$ to be in the logarithmic limit set of a more general semialgebraic set. If in addition $w \in (\R_{>0})^n$, then the corresponding quadratic module is stable. In particular, if $\In_w(I)$ is real radical for some $w \in (\R_{>0})^n$ then $\sum \R[x_1, \hdots, x_n]^2 +I$ is stable.  This provides a method for checking the conditions for stability given by Powers and Scheiderer \cite{SP}. \end{abstract}

\section{Introduction} \label{sec:intro}
Initial ideals can be seen as degenerations that retain useful information of the original ideal but often have simpler structure. The theory of Gr\"obner bases and much of computational algebraic geometry take advantage of this retention.  Tropical geometry uses initial ideals and degenerations of algebraic varieties to understand the combinatorial structure of an ideal (see \cite{trop book}). 
Here we will explore a property of initial ideals relevant to the real variety of an ideal, namely that an initial ideal is real radical. 

Because $\R$ is not algebraically closed, Gr\"obner basis techniques are not sufficient to algebraically characterize a real variety. This leads to the theory of sums of squares of polynomials. Many computations involving sums of squares can now be performed numerically with semidefinite programming, which is effective and deeply related to real algebraic geometry. For an introduction to sums of squares of polynomials and these connections, see \cite{PPSS}.

This paper has two main theorems, both stating consequences of an initial ideal being real radical. In Section \ref{tropical}, we introduce some useful constructions from tropical geometry and prove our first main result, which relates real radical initial ideals to the logarithmic limit set of real varieties and other semialgebraic sets. In particular, we show that a nonsingular point of $\V_{\R^*}(\In_w(I))$ ensures that the vector $w$ lies in the logarithmic limit set of $\V_{\R^*}(I)$.  Section \ref{sec:stable} deals with representations of polynomials as sums of squares modulo an ideal and more generally as elements of a preorder, which has implications for certain problems in real algebraic geometry and semidefinite programming. Our second main result gives conditions on a set of polynomials so that the preorder they generate is stable, as defined in \cite{SP}. 
In particular, if an initial ideal is real radical, there are degree bounds for the representation of polynomials as sums of squares modulo the ideal.  
These results can be better understood by embedding $\C^n$ in a weighted projective space $\proj^{(1,w)}$, which we discuss in Section~\ref{previous}.  In Section~\ref{sec:comp}, we will mention some of the current algorithms for computing real radicals and give some conditions under which these computations become more tractable.  We conclude by discussing the problem of determining the compactness of a real variety from its initial ideals. 

First we need to introduce notation and ideas from the theory of Gr\"obner bases and real algebraic geometry. For further background, see \cite{UGB} and \cite{PPSS}. 

We often use $\underline{x}^a$ to denote the monomial $x_1^{a_1}\hdots x_n^{a_n}$, $f(\underline{x})$ for $f(x_1, \hdots, x_n)$, and $\R[\underline{x}]$ for $\R[x_1, \hdots, x_n]$.
For $w \in \R^n$ and a polynomial $f(\underline{x}) = \sum_{a} f_{a}\underline{x}^a$, define
\[\deg_w(f) = \max\{w^Ta \;:\; f_a\neq 0\} \;\;\;\;\;\; \text{   and   }\;\;\;\;\;\In_w(f) = \sum_{a \;:\; w^Ta= \deg_w(f)}f_a\underline{x}^a.\] 
If $w = (1,1, \hdots, 1)$, we will drop the subscript $w$. 
For an ideal $I \subseteq \R[\underline{x}]$, define its \textbf{initial ideal}, $\In_w(I)$, as $\langle \In_w(f) \;:\; f \in I\rangle$.  For $w \in (\R_{\geq 0})^n$ we call $\{h_1, \hdots, h_s\} \subset I$ a \textbf{$w$-Gr\"obner basis} for $I$ if $\In_w(I) = \langle \In_w(h_1), \hdots, \In_w(h_s)\rangle$. Any $w$-Gr\"obner basis for $I$ generates $I$ as an ideal. 
For $F \subset \C$ and an ideal $I\subset \R[x_1, \hdots, x_n]$, let $\V_F(I)$ denote $\{p \in F^n \::\: f(p)=0 \;\;\;\forall f \in I\}$. 
\begin{definition}The \textbf{real radical} of an ideal $I$ is 
\[ \sqrt[\R]{I} \;\;: = \;\;\{ f \in \R[\underline{x}] \;:\; -f^{2m} \in \sum \R[\underline{x}]^2 +I \text{ for some }m \in \mathbb{Z}_+\},\] where $\sum \R[\underline{x}]^2 = \{ \sum h_i^2 \;:\; h_i \in \R[\underline{x}]\}.$
By the Positivstellensatz \cite[\S 2.2]{PPSS}, an equivalent characterization is
\[\sqrt[\R]{I} \;\;=\;\; \{f \in \R[\underline{x}]\;:\; f(p)=0 \;\;\;\forall p \in \V_{\R}(I) \}.\]
We call an ideal $I$ \textbf{real radical} if $\sqrt[\R]{I} =I$. \end{definition} 

\begin{prop}\label{subfan} If $\In_w(I)$ is real radical for some $w\in (\R_{\geq 0})^n$, then $I$ is real radical.\end{prop}

\begin{proof} If the set $\{\sum f_i^2 \in I \;:\; f_i \notin I \; \forall i\}$ is nonempty, then it has an element with minimal $w$-degree, $\sum_{i=1}^m f_i^2$. Then for $\mathcal{A} = \argmax_{i\in [m]}\{\deg_w(f_i)\}$, the polynomial $\sum_{i \in \mathcal{A}} \In_w(f_i)^2$ is in $\In_w(I)$. As $\In_w(I)$ is real radical, we have $\In_w(f_i) \in \In_w(I)$ for $i \in \mathcal{A}$. For $i \in \mathcal{A}$, let $g_i \in I$ so that $\In_w(f_i) = \In_w(g_i)$ and consider the following polynomial:
\[\sum_{i \in \mathcal{A}} (f_i - g_i)^2 + \sum_{i\in [m] \backslash \mathcal{A}}f_i^2 \;\;\;=\;\;\; \sum_{i \in [m]}f_i^2 + \sum_{i \in \mathcal{A}}g_i(-2f_i+g_i) \;\;\in \;I.\]
This is a sum of squares in $I$ with strictly lower $w$-degree than $\sum_i f_i^2$. As $f_i-g_i \notin I$ for all $i \in \A$, this contradicts our choice of  $\sum_i f_i^2$.
\end{proof}

A \textbf{cone} is a subset of $\R^n$ that is closed under addition and multiplication by nonnegative scalars. A cone is a \textbf{rational polyhedral} cone if it is the intersection of finitely many halfspaces defined by linear inequalities with rational coefficients. We say that a finite collection of rational polyhedral cones $\Delta$ is a \textbf{fan} if is closed under intersection and taking faces. The \textbf{support} of a fan $\Delta$, denoted $|\Delta|$, is the union of the cones in $\Delta$.
An ideal $I$ defines an equivalence relation on $\R^n$, by letting $w \sim v$  whenever $\In_{w}(I) = \In_{v}(I)$. One can show that there are only finitely many equivalence classes, and for homogeneous ideals each equivalence class is a relatively open rational polyhedral cone \cite{UGB}.  Together the closures of these cones form a fan called the \textbf{Gr\"obner fan} of $I$, denoted $\Delta_{Gr}(I)$. Using the software {\tt GFan} \cite{gfan}, one can actually compute the Gr\"obner fan of a homogeneous ideal. If $I$ is not homogeneous, then we homogenize by adding a new variable. 
For any $f \in \R[x_1, \hdots, x_n]$ and ideal $I \subset \R[x_1, \hdots, x_n]$, let $\f(x_0, x_1, \hdots, x_n)$ denote its homogenization $(x_0)^{\deg(f)}f(x_1/x_0, \hdots, x_n/x_0)$ and $\I = \langle \f \;:\; f\in I \rangle$.  We can define $\Delta_{Gr}(I)$ to be the cones in $\R^n$ obtained by intersecting cones of $\Delta_{Gr}(\overline{I})$ with the plane $\{w_0 =0\}$. The closure of an equivalence class $[w] = \{v \in \R^n \;:\; \In_v(I) = \In_w(I)\}$ will be a union of cones in $\Delta_{Gr}(I)$.  

Let $\Delta_{\R {\rm ad}}(I)$ denote the subset of $\Delta_{Gr}(I)$ corresponding to real radical initial ideals. That is, 
\[\Delta_{\R {\rm ad}}(I)\;\;:=\;\; \{\sigma \in \Delta_{Gr}(I)\;:\; \In_w(I) \text{ is real radical for }w\in {\rm relint}(\sigma) \}, \]
where ${\rm relint}(\sigma)$ denotes the relative interior of the cone $\sigma$.
For homogeneous $I$ and any $w \in \R^n$, $\In_w(I) = \In_{w+(1, \hdots, 1)}(I)$, thus we can assume $w \in (\R_{\geq 0})^n$. 
 Then Proposition \ref{subfan} ensures that for homogeneous $I$, $\Delta_{\R {\rm ad}}(I)$ is closed under taking faces, meaning that it is actually a subfan of $\Delta_{Gr}(I)$.\\ 
 
Sums of squares of polynomials are used to approximate the set of polynomials that are nonnegative on $\R^n$. We can use sums of squares modulo an ideal, $\sum \R[\underline{x}]^2+I$, to instead approximate the set of polynomials that are nonnegative on its real variety, $\V_{\R}(I)$. These methods extend to any basic closed semialgebraic set $\{x \in \R^n \;:\; g_1(x)\geq 0, \hdots, g_s(x) \geq 0\}$ by considering the \textit{quadratic module} or \textit{preorder} generated by $g_1, \hdots, g_s$.    

Formally, given a ring $R$ (e.g. $\R[\underline{x}]$, $\R[\underline{x}]/I$), we call $P \subset R$ a \textbf{quadratic module} if $P$ is closed under addition, multiplication by squares $\{f^2\;:\; f\in R\}$, and contains the element 1. If in addition, $P$ is closed under multiplication then we call $P$ a \textbf{preorder}.
 For $g_1, \hdots, g_s \in R$, we use $QM(g_1, \hdots, g_s)$ and $PO(g_1, \hdots, g_s)$ to denote the quadratic module and preorder generated by $g_1, \hdots, g_s$ (respectively), where $R$ will be inferred from context: 
\begin{align}
QM(g_1, \hdots, g_s) &= \biggl\{ \;\sigma_0+\sum_{i=1}^s g_i\sigma_i \;:\; \sigma_i \in \sum R^2 \text{ for } i=0,1,\hdots, s\;\biggl\}, \text{  and} \nonumber \\
PO(g_1, \hdots, g_s) &= \biggl\{\;\sum_{e \in \{0,1\}^s} g_1^{e_1}\hdots g_s^{e_s}\sigma_e \;:\;\sigma_e \in \sum R^2 \text{ for }e\in\{0,1\}^s\;\biggl\}.\nonumber \end{align}

There are advantages to both of these constructions. Preorders are often needed to obtain geometrical results (such as Proposition \ref{PO basis} and Theorem \ref{zariski}). From a computational point of view though preorders are less tractable than quadratic modules because the number of terms needed to represent an element is exponential in the number of generators. 

For any set of polynomials $S \subset \R[\underline{x}]$, let $K(S)$ denote the subset of $\R^n$ on which all the polynomials in $S$ are all nonnegative. That is, $K(S)=\{x \in \R^n : f(x) \geq\nolinebreak0 \;\; \forall f \in S\}$. Note that for $g_1, \hdots, g_s \in \R[\underline{x}]$, $K(\{g_1, \hdots, g_s\})$ and $K(PO(g_1, \hdots, g_s))$ are equal subsets of $\R^n$. If $P = \sum \R[\underline{x}]^2+I$, then $K(P) =K(I)= \V_{\R}(I)$. 

In this paper we are especially interested in semialgebraic sets that are not full-dimensional, that is, they are real varieties or contained in real varieties. The corresponding quadratic modules and preorders will contain a nontrivial ideal. For ease of notation, let $QM(g_1, \hdots, g_s; I)$ and $PO(g_1, \hdots, g_s;I)$ denote $QM(g_1, \hdots, g_s, \pm h_1, \hdots, \pm h_t)$ and $PO(g_1, \hdots, g_s, \pm h_1, \hdots, \pm h_t)$ respectively, where $I = \langle h_1, \hdots, h_t \rangle$. Then
\[QM(g_1, \hdots, g_s;I)=QM(g_1, \hdots, g_s)+I\;\;\;\;\text{ and }\;\;\;\;PO(g_1, \hdots, g_s;I) =PO(g_1, \hdots, g_s)+I. \label{QMPOdef}\]

For an ideal $I \subset \R[\underline{x}]$, the preorder $\sum \R[\underline{x}]^2+I$ has nice properties if $\In_w(I)$ is real radical, as we'll see in Section \ref{sec:stable}. To extend these properties to more general quadratic modules and preorders, we need the following definitions. 
For $g_1, \hdots, g_s \in \R[\underline{x}]$ and an ideal $I \subset \R[\underline{x}]$, say that $g_1, \hdots, g_s$ form a \textbf{quadratic module basis} (QM-basis) with respect to $I$ if for all $y_{ij} \in \R[\underline{x}]$,
\[ \sum_j y_{0j}^2 + \sum_{i,j} g_i y_{ij}^2 \in I \;\;\; \Rightarrow\;\;\; y_{ij} \in I \;\;\;\forall\; i,j.\;\;\;\;\;\;\]
Similarly, say that $g_1, \hdots, g_s$ form a \textbf{preorder basis} (PO-basis)\label{def:pobasis} with respect to $I$ if for all $y_{ej} \in \R[\underline{x}]$,
\[\;\;\;\;\;\; \sum_{e, j} g^e y_{ej}^2 \in I \;\;\; \Rightarrow\;\;\;y_{ej} \in I \;\; \;\forall \;e,j,\]
where $g^e = g_1^{e_1}\hdots g_s^{e_s}$ for $e \in \{0,1\}^s$.  Any PO-basis with respect to $I$ is also a QM-basis with respect to $I$, but the converse is not generally true. We see that $\emptyset$ is a PO-basis if and only if $I$ is real radical. As shown in Example \ref{ex:PO basis}, not every preorder has generators which form a PO-basis with respect to $P \cap -P$, though many do. The following proposition gives a geometric condition for elements forming a PO-basis.

\begin{prop} \label{PO basis} For $g_1, \hdots, g_s \in \R[\underline{x}]$ and an ideal $I \subset \R[\underline{x}]$, the polynomials $g_1, \hdots, g_s$ form a PO-basis with respect to $I$ if and only if the set $\{x \in \V_{\R}(I):g_i(x)>0 \text{ for all } i=1,\hdots, s\}$ is Zariski-dense in $\V_{\R}(I)$ and the ideal $I$ is real radical. \end{prop}
\begin{proof}
Let $P = PO(g_1, \hdots, g_s;I)$. By a slight abuse of notation, let \[K_+(P) = \{x \in \V_{\R}(I)\;:\; g_i(x) >0 \text{ for all } i=1,\hdots, s\} \subset K(P).\] 

($\Leftarrow$) 
Suppose $\sum_{e,j}g^e y_{ej}^2 \in I$. For all $x \in K_+(P)$,  we have $\sum_{e,j}g^e(x)y_{ej}(x)^2 =0$. Since $g^e(x)>0$, we have that  $y_{ej}(x)=0$ for all $x \in K_+(P)$. Then by Zariski-denseness of $K_+(P)$, we have $y_{ej} = 0$ on $\V_{\R}(I)$. Since $I$ is real radical, this gives us that $y_{ej} \in I$. 

($\Rightarrow$) 
We have $\sum y_j^2 \in I$ implies $y_j \in I$, so $I$ is real radical.
It is not difficult to check that $P \cap -P = I$. %(easy to check)\[ \sum g^e\sigma_e + h = -\sum g'^e\sigma_e' + h'\]\[ \sum g^e\sigma_e +\sum g'^e\sigma_e' =  h'-h\] So $\sum g^e\sigma_e +\sum g'^e\sigma_e' \in I$, meaning that $\sigma_e \in I$ for all $e$. Thus $f \in I$. 
Suppose $f = 0$ on $K_+(P)$. Then $\hat f = f^2 \prod_i g_i = 0$ on $K(P)$. By the Positivstellensatz \cite[Ch. 2]{PPSS}, 
$-\hat f^{2m} \in P \cap -P = I$ for some $m \in \N$. Since $I$ is real radical, we have $\hat f  = f^2 \prod_i g_i \in I$. Because $g_1, \hdots, g_s$ form a PO-basis with respect to $I$, we see that $f \in I$. Thus $K_+(P)$ is Zariski-dense in $\V_{\R}(I)$. 
\end{proof}
\begin{example}\label{ex:PO basis} Here is an example of a preordering $P$ where $P \cap -P$ is real radical %and $P$ is stable,
but no set of generators for $P$ form a PO-basis with respect to $P \cap -P$. Consider $P = PO(x,y ; \langle xy \rangle)$. Notice that $K(P) \subset \R^2$ is the union of the nonnegative $x$ and $y$ axes. We have $P \cap -P = \langle xy \rangle$, which is real radical. 

%It is not difficult to check that $P$ is stable. Suppose $f(x,y) = \sum a_i^2+ x \sum b_i^2+ y \sum c_i^2 +xy \cdot d$. If the degree of $f(x,y)$ is less than that of the top degree of terms on the right-hand side, then some of the initial terms on the right must cancel. Note that w.l.o.g. we can assume that $b_i \in \R[x]$ and $c_i \in \R[y]$. \[\sum_{i \in A_1} \In(a_i)^2+ x \sum_{i \in A_x} \In(b_i)^2+ y \sum_{i \in A_y} \In(c_i)^2 +xy \cdot\In(d) =0, \] where at least two of the $A_j$ are nonempty.  Plug in $(x,y) = (x,0)$. Then $\sum_{i \in A_1} \In(a_i)(x,0)^2+ x \sum_{i \in A_x} \In(b_i)(x,0)^2=0$. As the top terms of $x$ in both sums must be positive, this can only occur if $A_1$ and $A_x$ are empty. Similarly plugging in $(x,y)=(0,y)$ shows that $A_1$ and $A_y$ are empty. 

Now suppose $g_1, \hdots, g_s$ form a PO-basis for $\langle xy \rangle$. We will show $x \notin PO(g_1, \hdots, g_s; \langle xy \rangle)$, meaning $P \neq PO(g_1, \hdots, g_s; \langle xy \rangle)$. By Proposition \ref{PO basis}, there exists $p \in \R_{\geq 0}$ so that $g_i(0,p) >0$ for all $i=1, \hdots, s$. Then for all $q$ in a small enough neighborhood of $p$, $g_i(0,q)>0$. Now suppose $x = \sum_e g^e\sigma_e + xyh$ for some $\sigma_e \in \sum \R[x,y]^2$ and $h \in \R[x,y]$.  Plugging in $(0,q)$ gives $\sum_e g^e(0,q) \sigma_e(0,q)=0$. As $g^e(0,q)>0$ and $\sigma_e(0,q)\geq 0$, this implies that $\sigma_e(0,q)=0$ for every $e$. As this occurs for every $q$ in a neighborhood of $p$, we have that $\sigma_e \in \langle x \rangle$ for every $e$. Because $\sigma_e$ is a sum of squares, we actually have $\sigma_e \in \langle x^2 \rangle$.  This means that $\sum_e g^e\sigma_e + xyh$ and its partial derivative $\partial/\partial x$ vanish at $(0,0)$, which is not true of $x$. Thus $x \notin PO(g_1, \hdots, g_s;\langle xy \rangle)$ and $P \neq PO(g_1, \hdots, g_s; \langle xy \rangle)$. 
\end{example}

We will need PO-bases for the geometric result of Theorem \ref{zariski} but only QM-bases for the more algebraic result of Theorem \ref{stable}.

\section{Real Tropical Geometry}\label{tropical}
Tropical geometry is often used to answer combinatorial questions in algebraic geometry. For an introduction, see \cite{trop book}, \cite{trop2}.  We investigate analogous constructions for real algebraic geometry. 

\begin{definition}Given an ideal $I$, define its \textbf{tropical variety} as
\[\Trop(I)\;: = \;\{w \in \R^n \;\;:\; \In_w(I) \text{ contains no monomials}\}.\]
\end{definition}
Soon we will see an equivalent definition involving logarithmic limit sets.

\begin{definition} The \textbf{logarithmic limit set} of a set $V \subset (\R_+)^n$ is defined to be
\begin{align} \LL(V) := \;\;&\lim_{t\rightarrow 0} \log_{1/t}(V) \nonumber\\
= \;\;&\{x \in \mathbb{R}^n \;:\; \text{ there exist sequences } y(k) \in V \text{ and }t(k) \in (0,1), \nonumber \\ 
&\;\;t(k) \rightarrow 0   \text{ with }  \log_{\frac{1}{t(k)}} (y(k)) \rightarrow x \}. \nonumber \end{align} \end{definition}
The operations $\log$ and $|\cdot|$ on $\R^n$ are taken coordinate-wise. For a subset $V$ of $(\R^*)^n$ or $(\C^*)^n$, we will use $\LL(V)$ to denote $\LL(|V|)$, where $|V| = \{|x| \;:\; x \in V\}$. 
A classic theorem of tropical geometry states that $\Trop(I) = \LL(\V_{\C}(I))$, \cite[\S1.6]{trop book}. \\

Two analogs have been developed for varieties of $\R_+$, which can easily be extended to $\R^*$. 
First is the \textbf{positive tropical variety}, $\Trop_{\R_+}$, studied by Speyer and Williams for Grassmannians  \cite{real trop} and implicit in Viro's theory of patchworking \cite{viro}:
\[\Trop_{\R_+}(I) =\{w \in \R^n \;:\; \In_w(I) \text{ does not contain any nonzero polynomials in } \R_+[x_1, . . . , x_n]\}. \]
To extend this to $\R^*$, we simply take the union over the different orthants of $\R^n$. For each $\pi \in \{-1,1\}^n$, define the ideal $\pi\cdot I = \{f(\pi_1x_1, \hdots, \pi_n x_n) \;:\; f \in I\}$. Then we can define
\[\Trop_{\R^*}(I) \;\;= \bigcup_{\pi \in \{-1,1\}^n}\Trop_{\R_+}(\pi \cdot I),\]
which we'll call the \textbf{real tropical variety} of $I$.  By Proposition~\ref{ET}, $\Trop_{\R^*}(I)$ is the support of smallest subfan of $\Delta_{Gr}(I)$ containing $\{w\;:\; \V_{\R^*}(\In_w(I)) \neq \emptyset \}$.
On the other hand, $\LL(\V_{\R^*}(I))$ is studied by Alessandrini \cite{log limit}, who shows that \[\LL(\V_{\R^*}(I)) \subseteq \Trop_{\R^*}(I).\]
Unfortunately, it is not always possible to describe $\LL(\V_{\R^*}(I))$ solely in terms of initial ideals (see the example on page \pageref{dissonance}). While our results below involve $\LL(\V_{\R^*}(I))$, $\Trop_{\R^*}(I)$ and $\Trop(I)$ are much more practical for computation. 
Now we are ready to present the connection between real radical initial ideals and these tropical constructions. 

\begin{thm} Let $g_1, \hdots, g_s \subset \R[\underline{x}]$, an ideal $I \subset \mathbb{R}[\underline{x}]$, and $w\in \Trop(I)$. 
If $ \In_w(I)$ is real radical and $\In_w(g_1), \hdots, \In_w(g_s)$ form a preorder basis with respect to $\In_w(I)$, then $w\in \LL(K_{\R^*})$, where $K_{\R^*} = \{x\in \V_{\R^*}(I) \;:\; g_i(x)\geq 0 \;\;\forall \;i=1,\hdots, s\}$.  
\label{zariski} \end{thm}

This provides the first inner approximation of $ \LL(K_{\R^*})$ in terms of $\In_w(I)$ and $\In_w(g_i)$. In fact, Lemma \ref{R^n} gives a stronger such inner approximation, discussed in Remark~\ref{rem:primDecop}.  Also Theorem~\ref{zariski} shows that if there exists $w \in \Trop(I) \backslash (\R_{\leq 0})^n$ such that $\In_w(I)$ is real radical and $\In_w(g_1), \hdots, \In_w(g_s)$ forms a PO-basis with respect to $\In_w(I)$, then the corresponding semialgebraic set $K(\{g_1, \hdots, g_s\}\cup I)$ is not compact. For further discussion of compactness, see Corollary~\ref{cor:compact}.

\begin{lemma}\label{compact}
For a polynomial $g \in \mathbb{R}[\underline{x}]$, $w \in \mathbb{R}^n$, and a compact set $S \subset (\mathbb{R}^*)^n$, if $\In_w(g) > 0$ on $S$, then there exists $c_0>0$ so that for every $c>c_0$,  $g(c^w\cdot S) \subset (0,\infty)$, where $c^w\cdot x = (c^{w_1}x_1, \hdots, c^{w_n}x_n)$     and $c^w\cdot S = \{c^w\cdot x\;:\;x \in S\}$.
\end{lemma}
\begin{proof}Write $g = g_d+\hdots +g_{d'}$ where $d>\hdots>{d'}$ and $g_k$ is $w$-homogeneous with $w$-degree $k$. So $\deg_w(g)=d$ and $\In_w(g) =g_d$.  Let $a= \argmin_{x \in S}|g_d(x)|$ and  $b_j = \argmax_{x \in S}|g_j(x)|$ for $j\neq d$. As $\In_w(g) >0$ on $S$, we have $|g_d(a)| = |\In_w(g)(a)| \neq 0$. Then
\[\lim_{c\rightarrow \infty} \frac{\sum_{j\neq d} |g_j(c^w\cdot b_j)|}{|g_d(c^w\cdot a)|}\; = \;\lim_{c\rightarrow \infty} \frac{\sum_{j\neq d}| c^jg_j(b_j)|}{|c^d g_d(a)|} \;= \;\frac{1}{|g_d(a)|}\lim_{c\rightarrow \infty} \sum_{j\neq d} \frac{c^j}{c^d} |g_j(b_j)|\; =\;0.\]
Thus there exists $c_0>0$ so that for every $c>c_0$ and for any $x \in S$  
\[ |g_d(c^w \cdot x)| \;\geq \;|g_d(c^w \cdot a)| \;> \;\sum_{j\neq d}|g_j(c^w \cdot b_j)|\; \geq\; \sum_{j \neq d}|g_j(c^w \cdot x)|.\]
Since $c^dg_d(x) = g_d(c^w\cdot x) >0$, we have $g(c^w\cdot x) \geq g_d(c^w\cdot x) - \sum_{j\neq d}|g_j(c^w\cdot x)| >0$.
This gives that $g(c^w \cdot x)>0$. \end{proof}

\begin{lemma}\label{grad} 
Let $I \subset \R[\underline{x}]$ be an ideal and $w \in (\R_{\geq 0})^n$. Suppose $p \in \V_{\R}(\In_w(I))$ such that the vectors $\{\nabla \In_w(f_i)(p)\;:\;i=1,\hdots,m \}$ are linearly independent, where $\In_w(I) = \langle \In_w(f_1), \hdots, \In_w(f_m) \rangle$ and $f_i \in I$. Then there exists $c_0 \in \R_+$ and sequence $\{x_c\}_{c \in (c_0, \infty)}$ such that $\{c^w \cdot x_c\}_{c \in (c_0,\infty)} \subset \V_{\R}(I)$ and $x_c \rightarrow p$ as $c \rightarrow \infty$. 
\end{lemma}
\begin{proof}
The basic idea is that for a compact set $S$ and large enough $c>0$, polynomials act like their initial forms on $c^w \cdot S$. 
Since $w \in \R_{\geq 0}$ and $\In_w(I) = \langle \In_w(f_1), \hdots, \In_w(f_m)\rangle$, we have $I = \langle f_1, \hdots, f_m\rangle$.
If $m<n$, let $H$ denote the $m$-dimensional affine space $p+{\rm span}\{\In_w(f_i)(p)\;:\; i=1,\hdots,m\}$.

For every $\epsilon \in \mathbb{R_+}$, define the following affine transformation of an $m$-dimensional cube:
\[B(\epsilon) := \{x \in H \;:\;  |\nabla \In_w f_i(p) \cdot (x-p)| \leq \epsilon \;\;\forall \;i\leq m\},\]
with facets for each $i \leq m$, 
\begin{align}B_i^-(\epsilon)&:= \{x \in B(\epsilon) \;:\;  \nabla \In_w f_i(p) \cdot (x-p) = - \epsilon \}, \nonumber \\
B_i^+(\epsilon)&:= \{x \in B(\epsilon) \;:\;  \nabla \In_w f_i(p) \cdot (x-p) =  \epsilon \}.\nonumber\end{align}
Because $\nabla\In_w(f_i)(p) \neq \underline{0}$, for small enough $\epsilon>0$ we have that $\In_w f_i (B_i^-(\epsilon)) \subset (-\infty, 0)$  and $\In_w f_i (B_i^+(\epsilon)) \subset (0,\infty)$, for each $i \leq m$. By the compactness of $B(\epsilon)$ and $B_i^{\pm}(\epsilon)$ and Lemma~\ref{compact}, there is some $c_{\epsilon}>0$ so that for each $c > c_{\epsilon}$ and $i\leq m$, \[ f_i ( c^w\cdot B_i^-(\epsilon)) \subset (-\infty, 0)  \;\;\;\;\text{ and }\;\;\;\;  f_i (c^w\cdot B_i^+(\epsilon) ) \subset (0,\infty). \] 
 
We will show that for every $c > c_{\epsilon}$, there is $x_{\epsilon,c} \in  B(\epsilon)$ with $c^w\cdot x_{\epsilon,c} \in \V_{\R}(I)$. To do this, we use the Poincar\'e-Miranda Theorem, which is a generalization of the intermediate value theorem \cite{P-M}.
Let $J^m :=[0,1]^m \subset \mathbb{R}^m$ and for each $i\leq m$, denote
\[ J_i^- := \{x \in J^m \;:\;x_i=0\} \;\;\;\;\;\;\text{ and }\;\;\;\;\;\;  J_i^+:= \{x \in J^m \;:\;x_i=1\}.\]

\begin{PM*}[\cite{P-M}]\label{P-M} 
Let $\psi:J^m \rightarrow \mathbb{R}^m$, $\psi = (\psi_1, \hdots, \psi_m)$, be a continuous map such that for each $i\leq m$, $\psi_i(J_i^-) \subset  (-\infty, 0] $ and $\psi_i(J_i^+) \subset [0,\infty)$. Then there exists a point $y \in J^m$ such that $\psi(y)=\underline{0} = (0,\hdots, 0)$.
\end{PM*}
Define a homeomorphism $\phi_c:\mathbb{R}^m \rightarrow c^w\cdot H$ such that $\phi_c(J^m) = c^w \cdot B(\epsilon)$ and for $i\leq m$,
\[  \phi_c(J_i^-) = c^w \cdot B_i^-(\epsilon) \;\;\;\;\;\;\text{ and }\;\;\;\;\;\; \phi_c(J_i^+)=c^w\cdot B_i^+(\epsilon).\]
Let $\psi=f \circ \phi_c$, where $f = (f_1, \hdots, f_m)$. Then $\psi :J^m \rightarrow \R^m$ is continuous and for each $i\leq m$, 
\[\psi_i(J_i^-) = f_i(c^w \cdot B_i^-(\epsilon)) \subset (-\infty, 0) \;\;\;\;\;\;\text{ and }\;\;\;\;\;\; \psi_i(J_i^+) = f_i(c^w \cdot B_i^+(\epsilon)) \subset (0, \infty).\]
By the Poincar\'e-Miranda Theorem, there exists $y \in J^m$ so that $\psi(y)=\underline{0}$. Then let $ x_{\epsilon,c} = c^{-w}~\cdot~\phi_c^{-1}(y)$. This gives $c^w \cdot x_{\epsilon,c} \in c^w \cdot B(\epsilon) = \phi_c^{-1}(J)$, meaning $x_{\epsilon, c} \in B(\epsilon)$. Also, $f(c^w \cdot x_{\epsilon, c}) = \psi(y) =\underline{0}$, which implies $c^w \cdot x_{\epsilon, c} \in \V_{\R}(I)$.  

We have that for every $c >c_{\epsilon}$, there exists $x_{\epsilon,c} \in B(\epsilon)$ such that $c^w\cdot x_{\epsilon,c} \in \V_ {\R}(I)$. 
We will let $x_c = x_{\epsilon, c}$ for appropriately chosen $\epsilon$. Fix a $\epsilon_0>0$. 
For every $c$ such that $c>c_{\epsilon_0}$ there exists $\epsilon_c$ so that $c >c_{\epsilon}$. By increasing $c$, we may choose $\epsilon_c \rightarrow 0$.  Let $x_c = x_{\epsilon_c, c}$. 
For every $c>c_{\epsilon_0}$, $c^w \cdot x_c \in \V_{\R}(I)$. Moreover, because $\epsilon_c \rightarrow 0$ and $x_c=x_{\epsilon_c,c} \in B(\epsilon_c)$, we see that $x_c \rightarrow p$ as $c\rightarrow \infty$.  \end{proof} 

\begin{lemma}\label{R^n}  Let $I \subset \R[\underline{x}]$ and $w \in \R^n$. If $p \in \V_{\R}(\In_w(I))$ is nonsingular, then there exists a $c_0 \in \R_+$ and $\{x_c\}_{c \in (c_0, \infty)}$ such that $\{c^w \cdot x_c\}_{c \in (c_0,\infty)} \subset \V_{\R}(I)$ and $x_c \rightarrow p$ as $c \rightarrow \infty$. \end{lemma} 
\begin{proof}
We will first assume $w \in (\R_{\geq 0})^n$ and generalize later on. Let $\In_w(I)_p$ be the localization of $\In_w(I)$ at $p$, that is, $\In_w(I)_p = \{ \frac{g}{h} \;:\; g\in \In_w(I), \;h(p)\neq0\}$.  
Because $\V_{\R}(\In_w(I))$ is nonsingular at $p$, $\In_w(I)_p$ is a complete intersection \cite[\S II.8]{Hart}. Thus we can find $f_1, \hdots, f_m \in I$  %\in\R[\underline{x}] $% can we really make them real??? 
so that $\In_w(f_1), \hdots, \In_w(f_m)$ generate $\In_w(I)_p$ where $\dim(\In_w(I)_p)=n-m$. Because $\In_w(I)_p$ is nonsingular, we also have that the vectors $\nabla\In_w(f_1)(p) \hdots, \nabla\In_w(f_m)(p)$ are linearly independent. 
Now extend these to generators $\In_w(f_1), \hdots, \In_w(f_m), \In_w(f_{m+1}), \hdots, \In_w(f_r)$ of the ideal $\In_w(I)$, with $f_{m+1},\hdots, f_r \in I$. Because $\In_w(I)_p = \langle \In_w(f_1), \hdots, \In_w(f_m)\rangle$, there is a $w$-homogeneous $h$ so that $h(p)\neq0$ and $h\cdot \In_w(f_k) \in \langle \In_w(f_1), \hdots, \In_w(f_m)\rangle$ for every $k=m+1, \hdots, r$. 
 
Consider the ideal $I' = \langle f_1, \hdots, f_m, h\cdot f_{m+1}, \hdots, h\cdot f_r\rangle$. Since $\{f_1, \hdots, f_r\}$ is a $w$-Gr\"obner basis for $I$  and $h$ is $w$-homogeneous, one can check that $\{f_1, \hdots, f_m, h\cdot f_{m+1}, \hdots, h \cdot f_r\}$ forms a $w$-Gr\"obner basis for $I'$. By construction, \[\In_w(h \cdot f_k) = h\cdot \In_w(f_k) \in \langle \In_w(f_1), \hdots, \In_w(f_m)\rangle\] for $k=m+1, \hdots, r$, so in fact $\{f_1, \hdots, f_m\}$ forms a $w$-Gr\"obner basis for $I'$.  To summarize, we have $\In_w(I') = \langle \In_w(f_1), \hdots, \In_w(f_m)\rangle$ and $\nabla\In_w(f_1)(p), \hdots, \nabla\In_w(f_m)(p)$ linearly independent. 

Using Lemma \ref{grad}, we have $c_0 \in \R_+$ and $\{x_c\}_{c\in (c_0, \infty)}$ with $\{c^w\cdot x_c\}_{c \in (c_0,\infty)} \subset \V_{\R}(I')$ and $x_c \rightarrow p$ as $c\rightarrow \infty$. Because $h(p)\neq 0$ and $h$ is $w$-homogeneous, for large enough $c$, $h(c^w\cdot x_c)=c^{\deg_w(h)}h(x_c)\neq 0$. As $c^w\cdot x_c \in \V_{\R}(I')$, which is contained in $\V_{\R}(I)\cup \V_{\R}(h)$, we see that $c^w\cdot x_c \in \V_{\R}(I)$. This proves Lemma \ref{R^n} in the case $w \in (\R_{\geq 0})^n$. 

Now consider an arbitrary vector $w \in \R^n$. 
For any $f \in \R[x_1, \hdots, x_n]$, let $\f(x_0, x_1, \hdots, x_n)$ denote its homogenization $(x_0)^{\deg(f)}f(x_1/x_0, \hdots, x_n/x_0)$, and $\I = \langle \f \;:\; f\in I \rangle$.  Then $\V_{\R}(I) = \V_{\R}(\I) \cap \{x_0 = 1\}$.
For $f \in \R[\underline{x}]$ there is some $d \in \N$ so that $\In_{(0,w)}(\f) = (x_0)^d \overline{\In_w(f)}$. Using this, we can see that since $p \in \V_{\R}(\In_w(I))$ is nonsingular, we have $(1,p)$ is nonsingular in $\V_{\R}(\In_{(0,w)}(\I))$. 
% $ \partial/\partial x_j \In_{(0,w)}(\f) = z^{d_i}  \partial/\partial x_j \In_w(f)$ , so  there is some $a \in \R$ so that $\nabla \In_{(0,w)}(\f) (1,p) = (a, \nabla \In_w(f)(p))$. Then if the rank of $[\nabla \In_{w}(f_i)(p)] = r$, then the rank of $[\nabla \In_{(0,w)}(\f_i)(p,1)]$ is at least $r$ (since is contains the first as a submatrix)$.  

Choose $b \in \R_+$ so that $v:=(0,w)+b(1,\hdots, 1) \in (\R_{\geq 0})^{n+1}$. Since $\I$ is homogeneous, $\In_{(0,w)}(\I) = \In_v(\I)$ and we have that $(1,p)$ is nonsingular in $\V_{\R}(\In_v(\I))$. By the case $v \in (\R_{\geq 0})^{n+1}$ shown above, there exists $c_0 \in \R_+$ and $\{c^v\cdot y_c\}_{c\in(c_0, \infty)} \subset \V_{\R}(\I)$ so that $y_c \rightarrow (1,p)$ as $c \rightarrow \infty$.  As $\I$ is homogeneous and $c^v = c^bc^{(0,w)}$, we have that $\{c^{(0,w)}\cdot y_c\}_{c\in(c_0, \infty)} \subset \V_{\R}(\I)$. 
Because $y_c \rightarrow (1,p)$, by increasing $c_0$ if necessary we may assume that $(y_c)_0 \neq 0$ for all $c \in (c_0, \infty)$. This lets us scale $y_c$ to have first coordinate 1. Let $y_c' = (1/(y_c)_0)\cdot y_c := (1,x_c)$. Then $c^{(0,w)} \cdot y_c' = (1, c^w\cdot x_c) \in \V_{\R}(\I)$, giving us $c^w\cdot x_c \in \V_{\R}(I)$. In addition, as $(y_c)_0 \rightarrow 1$, we have that $x_c \rightarrow p$ as $ c \rightarrow \infty$. 
\end{proof}

Now we generalize to arbitrary semialgebraic sets using the notion of a PO-basis (page~\pageref{def:pobasis}):

\begin{lemma}\label{lim} Consider $g_1, \hdots, g_s \subset \R[\underline{x}]$, an ideal $I \subset \R[\underline{x}]$, and $w \in \R^n$. If $\In_w(g_1), \hdots, \In_w(g_s)$ form a preorder-basis with respect to $\In_w(I)$, then there is a Zariski-dense subset $U$ of $\V_{\R}(\In_w(I))$ so that for every $p \in U$, there exists $c_0 \in \R_+$ and $\{x_c\}_{c \in (c_0, \infty)} \subset \R^n$ with $\{c^w \cdot x_c\}_{c \in (c_0, \infty)} \subset K(\{g_1,\hdots, g_s\}\cup I)$ and $x_c \rightarrow p$ as $c \rightarrow \infty$. 
\end{lemma}

\begin{proof} 
As $\In_w(I)$ is real radical, it is radical. So we can write $\In_w(I)$ as an intersection of primes ideals $\mathfrak{a}_i$; $\In_w(I) = \cap_i \mathfrak{a}_i$ where for all $j$, $\mathfrak{a}_j \not\subseteq \cap_{i\neq j} \mathfrak{a}_i$. Since $\In_w(I)$ is real radical, we have that $\mathfrak{a}_i$ is real radical for all $i$. 
By  \cite[Thm 1.5.3]{Hart}, the singular points of $\V_{\C}(\mathfrak{a}_i)$ form a proper Zariski-closed subset. Because $\mathfrak{a}_i$ is real radical, $\V_{\R}(\mathfrak{a}_i)$ is Zariski-dense in $\V_{\C}(\mathfrak{a}_i)$, so singular points of $\V_{\R}(\mathfrak{a}_i)$ form a proper Zariski-closed subset of $\V_{\R}(\mathfrak{a}_i)$. The nonsingular points of $\V_{\R}(\mathfrak{a}_i)$, denoted $\Reg(\V_{\R}(\mathfrak{a}_i))$, therefore form a nonempty Zariski-open subset of $\V_{\R}(\mathfrak{a}_i)$. 
This implies that $\Reg(\V_{\R}(\In_w(I)))$ forms a Zariski-open, Zariski-dense subset of $\V_{\R}(\In_w(I))$. 

Let $K^{in}_+ = \{p \in \V_{\R}(\In_w(I)) \;:\; \In_w(g_i)(p)>0 \text{ for }i=1, \hdots, s \}$. By Proposition~\ref{PO basis}, $K^{in}_+$ is Zariski dense in $\V_{\R}(\In_w(I))$. Thus $K_+^{in}\cap \Reg(\V_{\R}(\In_w(I))$ is Zariski-dense in $\V_{\R}(\In_w(I))$. Let $U$ denote this intersection, $K_+^{in}\cap \Reg(\V_{\R}(\In_w(I))$, and consider a point $p \in U$. 
By Lemma~\ref{R^n}, there exist $c_0 \in \R_+$ and $\{c^w \cdot x_c\}_{c\in (c_0,\infty)} \subset \V_{\R}(I)$ such that $x_c\rightarrow p$ as $c \rightarrow \infty$. Because $p \in K^{in}_+$, $\In_w(g_i)(p)>0$. As $x_c \rightarrow p$, for large enough $c$, $\In_w(g_i)(x_c) >0$ for all $i=1, \hdots, s$. Using Lemma~\ref{compact}, we can find a $c_0'>c_0$ so that for all $c> c_0'$, $g_i(c^w\cdot x_c)>0$ for $i=1, \hdots, s$.  Thus for $c > c_0'$, we have $c^w \cdot x_c \in \V_{\R}(I) \cap \{x\;:\; g_i(x)\geq 0 \;\forall\; i=1, \hdots, s\} 
= K(\{g_1, \hdots, g_s\}\cup I)$. 
 \end{proof}

We now prove the main theorem of this section by removing the coordinate hyperplanes.

\begin{proof}[Proof of Theorem \ref{zariski}]  Since $w\in \Trop(I)$,  the monomial $x_1 \cdots x_n \notin \In_w(I)$. As $\In_w(I)$ is real radical, we can conclude that $x_1 \cdots x_n$ does not vanish on $\V_{\R}(\In_w(I))$. Thus $\V_{\R}(\In_w(I))\cap \V_{\R}(x_1 \cdots x_n)$ is a proper Zariski-closed subset of $\V_{\R}(\In_w(I))$, meaning $\V_{\R^*}(\In_w(I))$ is a non-empty Zariski-open subset of $\V_{\R}(\In_w(I))$. %Let $\In_w(I) = \cap_i \mathfrak{a}_i$ as in the proof of Lemma \ref{lim}

Let $U \subseteq \V_{\R}(\In_w(I))$ as given by Lemma \ref{lim}. As $\V_{\R^*}(\In_w(I))$ is nonempty, Zariski-open in $\V_{\R}(\In_w(I))$, we have that $U \cap (\R^*)^n$ is nonempty.  
So let $p \in U \cap (\R^*)^n$. By Lemma \ref{lim}, there exists $c_0 \in \R_+$ and $\{x_c\}_{c \in (c_0, \infty)} \subset \R^n$ such that $\{c^w \cdot x_c\}_{c \in (c_0, \infty)}  \subset K(\{g_1, \hdots, g_s\}\cup I)$ and $x_c \rightarrow p$ as $c \rightarrow \infty$. By increasing $c_0$ if needed, we can take $\{x_c\}_{c \in (c_0, \infty)} \subset (\R^*)^n$. 

Taking logarithmic limits with $c=1/t$ gives 
\[ \lim_{t\rightarrow 0} \log_{1/t}((1/t)^w \cdot x_{1/t} ) \;\;=\;\;   w+ \lim_{t\rightarrow 0} \log_{1/t}( x_{1/t} )\;\; =\;\; w.\]
From this, we conclude that $w$ lies in the logarithmic limit set, $\LL(K_{\R^*})$, of the set $K_{\R^*}$ where $K_{\R^*} = \{x\in \V_{\R^*}(I)\;:\; g_1(x)\geq0, \hdots, g_s(x)\geq 0\}$. \end{proof}

\begin{cor}\label{trops}For an ideal $I \subseteq \R[\underline{x}]$, we have the chain of inclusions: \[\Trop_{\R {\rm ad}}(I)\subseteq \LL(\V_{\R^*}(I)) \subseteq \Trop_{\R^*}(I) \subseteq \Trop(I),\] where $\Trop_{\R {\rm ad}}(I) = |\Delta_{\R {\rm ad}}(I)| \cap \Trop(I).$\end{cor}
All but the first inclusion are consequences of the  results in \cite{log limit} and \cite{real trop}, discussed above. Theorem \ref{zariski} with $\{g_1, \hdots,  g_s\} = \emptyset$  gives  $\Trop_{\R {\rm ad}} (I)\subseteq \LL(\V_{\R^*}(I))$. 

\subsection*{Example: Harmony} 
If an ideal has nice structure, we may see equality among our many different tropical constructions. For example, let $2 \leq d \leq n$ and consider the ideal $I_{d,n} \subset \R[x_{ij} \;:\; 1 \leq i \leq d, 1\leq j \leq n]$ generated by the determinants of the $\binom{n}{d}$ maximal minors of the $d \times n$ matrix 
\[	\begin{bmatrix} x_{11}&x_{12}&x_{13}&\hdots&x_{1n}\\ 
	 x_{21}&x_{22}&x_{23}&\hdots&x_{2n}\\ 
	\vdots&\vdots&\vdots&\ddots&\vdots\\
	 x_{d1}&x_{d2}&x_{d3}&\hdots&x_{dn}\\ 
\end{bmatrix}.\]
These polynomials form a \textbf{universal Gr\"obner basis} for $I_{d,n}$, by \cite{minors}, meaning that they form a $w$-Gr\"obner basis for every $w\in \R^n$. 

Any term of one of these polynomials is a square-free monomial with coefficient $\pm 1$. 
Every monomial initial ideal of $I_{d,n}$ is generated by square-free monomials and thus is real radical.   
Any initial ideal of $I_{d,n}$ has a further initial ideal that is monomial, so by Proposition~\ref{subfan}, every initial ideal of $I_{d,n}$ is real radical. That is, $|\Delta_{\R {\rm ad}} (I_{d,n})|= \R^n$. Intersecting with $\Trop(I_{d,n})$ and using Corollary \ref{trops} gives
\[\Trop_{\R {\rm ad}}(I_{d,n})= \LL(\V_{\R^*}(I_{d,n}))= \Trop_{\R^*}(I_{d,n}) =\Trop(I_{d,n}).\]

\subsection*{Example: Dissonance}\label{dissonance}
Unlike the positive tropical variety, the logarithmic limit set $\LL(\V_{\R^*}(I))$ cannot be characterized solely in terms of initial ideals. Consider $I = \langle (x-y-z)^4+(x-y-1)^2\rangle \subset \R[x,y,z]$. For this example, we'll demonstrate that each of the inclusions in Corollary \ref{trops} are strict: 
\[\Trop_{\R {\rm ad}}(I) \subsetneq \LL(\V_{\R^*}(I)) \subsetneq \Trop_{\R^*}(I) \subsetneq \Trop(I).\]

Because $I$ is principal, the tropical variety is given by the dual fan of the Newton polytope of $(x-y-z)^4+(x-y-1)^2$. Thus it is the union of the rays  $r_0 = (1,1,1)$, $r_1 = (0,0,-1)$, $r_2 = (0,-1,0)$, $r_3 = (-1,0,0)$ and the six 2-dimensional cones spanned by pairs of these rays. We'll use $\sigma_{ij}$ to denote the cone spanned by $r_i$ and $r_j$.  The corresponding monomial-free initial ideals are:
\begin{center}
\begin{tabular}{c|l}
$\sigma$ & $\In_{\sigma}(I)$\\
\hline 
$\underline{0}$ & $(x-y-z)^4+(x-y-1)^2$\\
$r_0$ & $(x-y-z)^4$\\
$r_1$ & $(x-y)^4 +(x-y-1)^2$\\
$r_2$ &$(x-z)^4+(x-1)^2$\\
$r_3$ &  $(y+z)^4 +(y-1)^2$\\
$\;$
\end{tabular} \hspace{2cm}
\begin{tabular}{c|l}
$\sigma$ & $\In_{\sigma}(I)$\\
\hline 
$\sigma_{01}$ & $(x-y)^4$\\
$\sigma_{02}$ & $(x-z)^4$\\
$\sigma_{03}$ & $(y+z)^4$\\
$\sigma_{12}$ & $x^4 +(x-1)^2$\\
$\sigma_{13}$ &  $y^4 +(y-1)^2$\\
$\sigma_{23}$ &$z^4+1$\\
\end{tabular}
\end{center}

To calculate $\Trop_{\R^*}(I)$, we use a theorem of Einsiedler and Tuncel \cite{ET} regarding the positive tropical variety, which easily extends to the following: 
\begin{prop}[\cite{ET}]\label{ET}  Fix an ideal $I \subset \R[\underline{x}]$. A vector $w\in \R^n$ lies in the real tropical variety of $I$, $\Trop_{\R^*}(I)$, if and only if there exists $v\in \R^n$ with $\V_{\R^*} (\In_v(\In_wI)) \neq \emptyset$.\end{prop}
Given $v, w \in \R^n$, for small enough $\epsilon >0$, $\In_v(\In_w I) = \In_{(w+\epsilon v)}(I)$. For a cone $\sigma \in \Delta_{Gr}$, let $\In_{\sigma}(I)$ denote $\In_w(I)$ for $w$ in the relative interior of $\sigma$. We can view $\In_{\sigma_{ij}}(I)$ as an initial ideal of $\In_{r_i}(I)$. To understand $\Trop_{\R^*}(I)$ we'll first look at the maximal cones, $\sigma_{ij}$. As each of their further initial ideals are monomial, we see that $\Trop_{\R^*}(I)$ contains $\sigma_{ij}$ if and only if the variety  $\V_{\R^*}(\In_{\sigma_{ij}}(I))$ is nonempty.  
Thus the two-dimensional cones of $\Trop_{\R^*}(I)$ are $\sigma_{01}, \sigma_{02}$, and $ \sigma_{03}$. Because each of the rays 
$r_0, r_1, r_2, r_3$ are contained in one of these cones, they are each in $\Trop_{\R^*}(I)$ as well. 

To calculate $\LL(\V_{\R^*}(I))$, note that the real variety of $I$ is $\{(x,y,z) \;:\; x-y=1,\;z=1 \}$. Because the real variety is contained in the plane $\{z=1\}$, we see that $\LL(\V_{\R^*}(I))$ is contained in the plane $\{z=0\}$. The intersection of $\Trop_{\R^*}(I)$ with the plane $\{z=0\}$ is the union of the three rays, $r_2, r_3$ and $(1,1,0)\subset \sigma_{03}$. The sequences $(1+t, t, 1)$, $(t,-1+t, 1)$, and $(1/t +1, 1/t,1)$ respectively verify that each of these rays is in $\LL(\V_{\R^*}(I))$. So  $\LL(\V_{\R^*}(I))$ is the union of the rays $r_2, r_3$ and $(1,1,0)$. Note that in this case, $\LL(\V_{\R^*}(I))$ is not a subfan of $\Trop_{\R^*}(I)$.

Finally, by the table listed above we see that $\Trop_{\R {\rm ad}}(I)$ is empty. To summarize, 

\begin{center}
\begin{tabular}{ll}
$\Trop_{\R {\rm ad}}(I)$ & $=\emptyset$\\
$\LL(\V_{\R^*}(I))$ & $= r_2 \cup r_3 \cup (1,1,0)$\\
$\Trop_{\R^*}(I) $ & $= \sigma_{01}\cup \sigma_{02} \cup \sigma_{03} $\\
$\Trop(I)$ & $= \sigma_{01}\cup \sigma_{02} \cup \sigma_{03} \cup \sigma_{12}\cup \sigma_{13} \cup \sigma_{23} $,\\
\end{tabular}\end{center}$\;$\\
and indeed, $\Trop_{\R {\rm ad}}(I) \subsetneq \LL(\V_{\R^*}(I)) \subsetneq \Trop_{\R^*}(I) \subsetneq \Trop(I)$.

\begin{remark}\label{rem:primDecop} Lemma~\ref{R^n} gives a stronger inner approximation of $\LL(\V_{\R^*}(I))$. We can characterize this algebraically as follows. Let $\In_w(I) = \cap_i \mathfrak{a}_i$
be the primary decomposition of $\In_w(I)$, see for example \cite[\S 3.3]{comm alg}. There exists a nonsingular point $p\in \V_{\R}(\In_w(I))$ if and only if for some $i$, $\mathfrak{a}_i$ is real radical \cite[Thm. 12.6.1]{PPSS}. If in addition, $\mathfrak{a}_i$ does not contain the monomial $x_1\hdots x_n$, 
%that is, $\V_{\R^*}(\mathfrak{a}_i)\neq \emptyset$, 
then there is such a point in $(\R^*)^n$ and $w \in \LL(\V_{\R^*}(I))$. Thus in Corollary~\ref{trops}, we may replace $\Trop_{\R {\rm ad}}(I)$ with $\{w\in \Trop(I)\;:\; \In_w(I) \text{ has a real radical primary component}\}.$ See Ex.~\ref{ex:primedec} for such a computation.\end{remark}
\section{Stability of sums of squares modulo an ideal}\label{sec:stable}

Sums of squares of polynomials are essential in real algebraic geometry. They are also made computationally tractable by methods in semidefinite programming. For an introduction to real algebraic geometry and its relation to sums of squares and semidefinite programming, see \cite{PPSS}. Many semidefinite optimization problems involve quadratic modules and their corresponding semialgebraic sets. 
Given a quadratic module $P = QM(g_1, \hdots, g_s)$ and semialgebraic set $K=K(P)$, there are some natural questions we can ask about quality of the approximation of $P$ to $\{f \in \R[\underline{x}] \;:\; f\geq 0 \text{ on } K\}$.
\begin{enumerate}
\item Does $P$ contain all $f$ that are positive on $K$?
\item Given $f \in P$, can we bound the degree of $\sigma_i \in \sum \R[\underline{x}]^2$ needed to represent $f = \sigma_0+ \sum_{i=1}^s \sigma_i g_i$ in terms of $\deg(f)$ only?
\end{enumerate}

We call a quadratic module \textit{stable} if the answer to (2) is yes. That is, 
\begin{definition} A quadratic module $P = QM(g_1, \hdots, g_s)\subset \R[\underline{x}]$ is \textbf{stable} if there exists a function $l:\mathbb{N}\rightarrow \mathbb{N}$ so that for all $f\in P$, there exist $\sigma_i \in \sum \R[\underline{x}]^2$ so that $f = \sum_{i}\sigma_ig_i$ where $g_0=1$, and for each $i$, $\deg(\sigma_ig_i) \leq l(\deg(f))$.
\end{definition}

For preorders, the answers to (1) and (2) are related and both heavily depend on the geometry of $K$. 
Scheiderer shows that for $K$ with dimension greater than one, a positive answer to (1) implies that $P$ is not stable \cite{Sch}. A complementary result of Schm\"udgen states that if $K(P)$ is compact, then $P$ contains any polynomial that is strictly positive on $K$ \cite{Schmudgen}. Thus for $K$ compact of dimension greater than one, the preorder $P$ is not stable. 

\begin{example} \label{ex:unstable}Consider the preorder $P = \sum \R[x,y,z]^2 +\langle x^4-x^3+y^2+z^2 \rangle$. Since the variety of $x^4-x^3+y^2+z^2$ is compact and two-dimensional, the preorder $P$ is not stable. Indeed, we can see that by Schm\"udgen's Theorem, $P$ must contain $x+\epsilon$ for every $\epsilon >0$. By inspection one can check that $x \notin P$.  For real radical ideals $I$, such as our example, and given $d$, the set $\{\sum_i g_i^2 +I \;:\; \deg(g_i)\leq d\}$ is a closed subset of the $\R$-vector space $\{f +I \;:\; \deg(f)\leq 2d\} \subset \R[\underline{x}] / I$, \cite[Lemma 4.1.4]{PPSS}. 
Thus there can be no $d$ for which $x+ \epsilon \in \{\sum_i g_i^2 +\langle x^4-x^3+y^2+z^2 \rangle\;:\; \deg(g_i)\leq d\}$ for every $\epsilon$ in a positive neighborhood of $0$. The degrees of polynomials verifying $x+\epsilon \in P$ must be unbounded as $\epsilon \rightarrow 0$.
\end{example}

In what follows, we give sufficient conditions to avoid this unstable behavior. The rest of this section will address question (2), but as we see from \cite{Sch} and \cite{Schmudgen} this has implications for the compactness of $K$ and the answer to (1).
In practice checking the stability of a quadratic module $P$ is difficult. Netzer gives tractable sufficient conditions when $K$ is full dimensional \cite{Netzer}. Here we present a complimentary result: sufficient conditions for stability in the case $K \subseteq \V_{\R}(I)$,  $P = PO(g_1, \hdots, g_s;I)$.

\begin{thm} Let $g_1, \hdots, g_s \in \R[\underline{x}]$ and $I \subset \R[\underline{x}]$ be a ideal.  If there exists $w \in (\mathbb{R}_{>0})^n$ so that $\In_w(I)$ is real radical and $\In_w(g_1), \hdots, \In_w(g_s)$ is a quadratic-module basis with respect to $\In_w(I)$, then $QM(g_1, \hdots, g_s;I)$ is stable.  \label{stable}
\end{thm}

\begin{lemma} 
Under the hypotheses of Theorem \ref{stable}, for any $f \in QM(g_1, \hdots, g_s;I)$, there exist $\sigma_i \in \sum\R[\underline{x}]^2$ with $f \equiv \sum_{i=0}^s\sigma_ig_i$ (mod $I$), where $g_0=1$,  and $ \max_i\{\deg_w(\sigma_ig_i)\}\leq \deg_w(f).$ \label{low degree}
\end{lemma}

\begin{proof}
Let $f \in \QM(g_1, \hdots, g_s;I)$. Then $f = \sum_{i=0}^s g_i\sum_j y_{ij}^2 +h$ for some $y_{ij} \in \R[\underline{x}]$ and $h\in I$. Let $d= \max_i\{\deg_w(g_iy_{ij}^2)\}$ and $\A = \{(i,j)\;:\; \deg_w(g_iy_{ij}^2) =d \}$.  Let $h'$ equal $h$ if $\deg_w(h)=d$ and 0 otherwise. Note that $h' \in I$ either way.  

Suppose $\deg_w(f)<d$.  This implies that top terms in the representation of $f$ must cancel, that is, $\sum_{(i,j)\in \A} \In_w(g_iy_{ij}^2)+\In_w(h')=0$. Since $\In_w(g_iy_{ij}^2) = \In_w(g_i)\In_w(y_{ij})^2$, we have
 \[\sum_{(i,j)\in \A} \In_w(g_i)\In_w(y_{ij})^2\;\;=\;\;-\In_w(h')\;\; \in \;\;\In_w(I).\]
Because $\In_w(g_1), \hdots, \In_w(g_s)$ is a QM-basis with respect to $\In_w(I)$, this implies $\In_w(y_{ij}) \in \In_w(I)$ for all $(i,j)\in \A$. Thus there is some $z_{ij}\in I$ so that $\In_w(z_{ij}) = \In_w(y_{ij})$. Let $\hat{y}_{ij} = y_{ij}-z_{ij}$ for $(i,j) \in \A$ and $\hat{y}_{ij}=y_{ij}$ for $(i,j) \notin \A$.  So $y_{ij} - \hat{y}_{ij} \in I$ for all $(i,j)$.  Then
\[f \;\;=\;\; \sum_i g_i \sum_j \hat{y}_{ij}^2 \;+\; h'' \]
where $h'' = h+\sum_{(i,j)\in \A} g_i\cdot (y_{ij}^2 -\hat{y}_{ij}^2) \in I$. Also, note that $\deg_w(\hat{y}_{ij}) <\deg_w(y_{ij})$ for all $(i,j) \in \A$. 
Then $\max_{(i,j)}\{\deg_w(g_i \hat{y}_{ij}^2)\}< d =\max_{(i,j)}\{\deg_w(g_i y_{ij}^2)\}$. If $\deg_w(f) < \max_{(i,j)}\{\deg_w(g_i \hat{y}_{ij}^2)\}$, then we repeat this process. Because the maximum degree drops each time and must be nonnegative, this process must terminate. This gives $f \equiv \sum_{i=0}^s \sigma_ig_i \text{ (mod } I)$ where $\deg_w(f) = \max_i\{\deg_w(\sigma_ig_i)\}$. 
 \end{proof}

%The converse of Lemma \ref{low degree} also holds, though it in unnecessary in the following proof of Theorem \ref{stable}.
\begin{proof}[Proof of Theorem \ref{stable}] As shown in \cite[\S 4.1]{PPSS}, $QM(g_1, \hdots, g_s;I)$ is stable in $\R[\underline{x}]$ if and only if $QM(g_1, \hdots, g_s)$ is stable in $\R[\underline{x}] / I$. Thus to show that  $QM(g_1, \hdots, g_s;I)$ is stable, it suffices to find $l:\mathbb{N} \rightarrow \mathbb{N}$ so that for every $f \in \sum QM(g_1, \hdots, g_s;I)$, there are $\sigma_i \in \sum\R[\underline{x}]^2$ with $f \equiv \sum_i\sigma_ig_i +I$ and $\deg(\sigma_ig_i) \leq l(\deg(f))$.

Let $f \in QM(g_1, \hdots, g_s;I)$. Let $\sigma_i$ be the polynomials given by Lemma \ref{low degree}. Note that for any $h \in \mathbb{R}[\underline{x}]$, 
\[ w_{min}\deg(h)\;\;\leq  \;\;\deg_w(h) \;\;\leq \;\;w_{max}\deg(h) ,\]
where $w_{max} = \max_i\{w_i\}$ and $w_{min} = \min_i\{w_i\}$.  Then
\[\;w_{min}\max_i\{\deg(\sigma_ig_i)\}\;\;\leq \;\;\max_i\{\deg_w(\sigma_ig_i)\} \;\;= \;\;\deg_w(f) \;\;\leq \;\;w_{max} \deg(f).\] 
So $ \max_i\{\deg(\sigma_ig_i)\} \leq \frac{w_{max}}{w_{min}} \cdot \deg(f).$
\end{proof}
The restriction $w\in (\R_{>0})^n$ cannot be weakened to $w\in (\R_{\geq 0})^n$. For example, consider the ideal $I$ of Example \ref{ex:unstable}. Since the ideal is real radical, the initial ideal with respect to the zero vector, $\In_{\underline{0}}(I)= I$, is real radical, but the preorder $\sum \R[\underline{x}]^2+I$ is not stable. 

\begin{cor}Let $I \subset \R[\underline{x}]$ be an ideal. If there exists $w  \in (\R_{>0})^n$ with $\In_w(I)$ real radical, then the preorder $\sum \R[\underline{x}]^2+I$ is stable. If $I$ is homogeneous and $\Delta_{\R {\rm ad}}(I) \neq \emptyset$, then $\sum \mathbb{R}[\underline{x}]^2 +I$ is stable.
\end{cor}
\begin{proof} The statement for general ideals follows from Theorem \ref{stable} with $\{g_1, \hdots, g_s\}=\emptyset$. If an ideal $I \subset \R[\underline{x}]$ is homogenous, then for every $v \in \R^n$, there exists $w\in (\R_{>0})^n$ so that $\In_v(I) = \In_w(I)$. Thus $\Delta_{\R {\rm ad}}(I) \neq \emptyset$ if and only if there exists $w \in (\R_{>0})^n$ for which $\In_w(I)$ is real radical.  \end{proof}

\section{Connections to compactification}\label{previous}

The results of both sections are best understood by embedding varieties of $\C^n$ into the weighted projective space $\proj^{(1,w)}$, which we will explain here. 

Consider $\In_w(I)$ for an ideal $I$ and $w \in \R^n$. Because the Gr\"obner fan is a rational polyhedral fan, there exists a vector $v \in \Z^n$ so that $\In_w(I) = \In_v(I)$. Thus we may replace $w \in (\R_{\geq 0})^n$ with $w\in \N^n$.  For a vector $w \in \N^{n}$, weighted projective space $\proj^{(1,w)}$ as a set is $\C^{n+1} \backslash \{0\}$ modulo the equivalence $(a_0, a_1, \hdots, a_n) \sim (ta_0, t^{w_1}a_1, \hdots, t^{w_n}a_n)$ for all $t \in \C^*$. We'll use $[a_0:a_1:\hdots:a_n]$ to denote the equivalence class of $(a_0, a_1, \hdots, a_n) \in \C^{n+1} \backslash \{0\}$. Varieties in $\proj^{(1,w)}$ are defined by the zero sets of $(1,w)-$homogeneous polynomials in $\C[x_0, \hdots, x_n]$. 
The \textit{real points} of $\proj^{(1,w)}$ are the points in the image of $\R^{n+1}\backslash \{0\}$ under the equivalence relation. In other words, $a \in \proj^{(1,w)}$ is an element of $\proj^{(1,w)}_{\R}$ if $a = [b_0:\hdots:b_n]$ for some $(b_0, \hdots, b_n)\in \R^{n+1}$. 

We embed $\C^n$ into $\proj^{(1,w)}$ by $(a_1, \hdots, a_n) \mapsto [1:a_1:\hdots:a_n]$. 
Let $I \subset \R[x_1, \hdots, x_n]$ be an ideal, and let $V$ denote the image of $\V_{\C}(I)$ under this map. Let $V_{\R}$ denote the image of $\V_{\R}(I)$. Let $\overline{V}$ and $\overline{V_{\R}}$ denote the closures of $V$ and $V_{\R}$ in the Zariski topology on $\proj^{(1,w)}$. 
For $f \in \C[\underline{x}]$, let $\overline{f}^w(x_0,x_1, \hdots, x_n) = x_0^{deg_w(f)}f(x_1/x_0^{w_1}, \hdots, x_n/x_0^{w_n})$. Then $\overline{f}^w$ is $(1,w)$-homogeneous and $\overline{f}^w(0,x_1, \hdots, x_n)=\In_w(f)(x_1, \hdots, x_n)$. 
For an ideal $I \subset \R[\underline{x}]$, let $\overline{I}^w = \langle \overline{f}^w \;:\;f\in I\rangle$. We see that $\overline{V}$ is cut out by $\overline{I}^w$.

For every $ t\in \C$, define $\overline{I}^w(t) = \{f(t, x_1, \hdots, x_n) \;:\; f \in \overline{I}^w\}$.  
The boundary of our embedding of $\C^n$ into $\proj^{(1,w)}$ is given by $\{a_0=0\} \cong \proj^w$. Thus $\overline{V} \backslash V$ is cut out by $\overline{I}^w(0)$ in $\{a_0=0\} \cong \proj^w$. Since $\overline{f}^w(0,x_1, \hdots, x_n) = \In_w(f)$, we see that $\overline{I}^w(0) = \In_w(I)$. So $\overline{V} \backslash V$ is cut out by $\In_w(I)$ in $\{a_0=0\} \cong \proj^w$.

\begin{example} \label{ex:Vbar} Note that  $\overline{(V_{\R})} \subset (\overline{V})_{\R}$, but in general we do \textit{not} have equality. For example, consider $I = \langle (x-y)^4+x^2\rangle$ and $w = (1,1)$. Since $\V_{\R}(I) = \{(0,0)\}$, we have $V_{\R} = \{[1:0:0]\}$. But $\overline{V}$ is cut out by $(x-y)^4+x^2t^2=0$. This gives
\[\overline{(V_{\R})}  = \{[1:0:0]\}, \;\;\;\text{  and  }\;\;\; (\overline{V})_{\R} = \{[1:0:0] , [0:1:1]\}.\] We'll see that if $\In_w(I)$ is real radical, then $\overline{(V_{\R})} = (\overline{V}_{\R})$. \end{example}

Consider the preordering $P = PO(g_1, \hdots, g_s;I)$ and the semialgebraic set $K = K(P)$. Suppose $\In_w(g_1), \hdots, \In_w(g_s)$ form a PO-basis with respect to $\In_w(I)$.  Consider $U$,$\;p \in U$ and $\{x_c\}_{c \in (c_0, \infty)}$ as given by Lemma \ref{lim}. %%%%%%%%%%%%%%%

For $y \in \C^n$ and $a \in \C$, let $[a:y]$ denote $[a:y_1:\hdots:y_n]$ in $\proj^{(1,w)}$. By embedding $K$ into $\proj^{(1,w)}$, we see that 
\[[1:c^w\cdot x_c] \;= \;[c^{-1}:x_c] \;\in \;K.\]
As $c \rightarrow \infty$, we have $(c^{-1}, x_c) \rightarrow (0,p)$ in $\R^{n+1}$. Thus $[0:p] \in \overline{K} \subset \proj^{(1,w)}$, where $\overline{K}$ is the closure of $K$ in the Euclidean topology on $\proj^{(1,w)}$.
This shows that for every $p \in U$, $[0:p] \in \overline{K}$. 

Recall that $\overline{V} \backslash V$ is cut out by $\In_w(I)$ in $\{a_0=0\}\cong \proj^w$. As $U$ is Zariski-dense in $\V_{\R}(\In_w(I))$ and $\In_w(I)$ is real radical, we have that $U$ is Zariski-dense in $\V_{\C}(\In_w(I))$. Together with $[0:p] \in \overline{K}$ for all $p \in U$, this gives that $\overline{K} \backslash K$ is Zariski-dense in $\overline{V} \backslash V$. 

\begin{remark} \label{dense}The conditions for stability given in Theorem \ref{stable} imply that $\overline{K}\backslash K$ is Zariski-dense in $\overline{V} \backslash V$ when we embed $K$ and $V$ into $\proj^{(1,w)}$. \end{remark}
 This very closely resembles the conditions for stability given in \cite{SP}.
After introducing the notion of stability in \cite{SP}, Powers and Scheiderer give the following general sufficient condition for stability of a preorder. 

\begin{thm}[Thm. 2.14, \cite{SP}] \label{PS} Suppose $I \subset \R[\underline{x}]$ is a radical ideal and $V = \V_{\C}(I)$ is normal. Let $P$ be a finitely generated preorder with $K(P) \subseteq V_{\R}$. Assume that $V$ has an open embedding into a normal complete $\R$-variety $\overline{V}$ such that the following is true: For any irreducible component $Z$ of $\overline{V}\backslash V$,  the subset $\overline{K}\cap Z_{\R}$ of $Z_{\R}$ is Zariski dense in $Z$, where $\overline{K}$ denotes the closure of $K$ in $\overline{V}_{\R}$. Then the preorder $P$ is stable. 
\end{thm}

By Remark \ref{dense}, the conditions in Theorem \ref{stable} give a specific compactification that (mostly) satisfies the conditions of Theorem \ref{PS}, namely embedding $V$ into $\proj^{(1,w)}$. Note Theorem \ref{stable} has no normality requirements, and even when the original variety $V$ is normal, its compactification $\overline{V}$ might not be. While Theorem \ref{stable} is less general, it has the advantage of having no normality requirements and providing a concrete method of ensuring stability.

\section{Computation}\label{sec:comp}

In this section, we discuss tractable cases and possible methods for computing $\Delta_{\R {\rm ad}}$ and give sufficient conditions for the compactness and non-compactness of a real variety. 

To make use of Theorems \ref{zariski} and \ref{stable} for a given ideal $I \subset \R[\underline{x}]$, we need to verify that their hypotheses are satisfied. Thus we seek
to compute $\Trop_{\R {\rm ad}}(I)$ and determine whether the set $|\Delta_{\R {\rm ad}}(I)| \cap (\R_+)^n$ is nonempty. More generally, we would like to compute $\Delta_{\R {\rm ad}}(I)$.
Each of these tasks involves checking whether initial ideals are real radical. In general, checking whether an ideal is real radical is difficult, so calculating $\Delta_{\R {\rm ad}}(I)$ and $\Trop_{\R {\rm ad}}(I)$ will be difficult as well. Initial ideals often have simpler form, making it tractable to partially compute $\Delta_{\R {\rm ad}}(I)$. Here we will discuss some types of ideals for which these calculations are more tractable and present some general heuristics for computation in the general case. 

One case in which computing $\Delta_{\R {\rm ad}}(I)$ is more tractable is when $I$ is principal. Consider $I = \langle f \rangle.$ As in the general case, maximal cones of $\Delta_{Gr}(I)$ correspond to initial monomials of $f$ and belong to $\Delta_{\R {\rm ad}}(I)$ if and only if these monomials are square-free. Next we can consider cones of codimension one, which are dual to edges of the Newton polytope of $f$, $N\!P(f)$.  Because an edge is one-dimensional, the initial form to which it corresponds can be thought of as a polynomial in one variable. Specifically, suppose $a = (a_1, \hdots, a_n)$ and $b = (b_1, \hdots, b_n)$ are the vertices of an edge of $N\!P(f)$, with dual cone $\sigma \subset \Trop(f)$. Then the Newton polytope of $\In_{\sigma}(f)$ is the edge with endpoints $a$ and $b$.  If for some $i \in \{1,\hdots, n\}$ both $a_i \geq 2$ and $b_i \geq 2$, then $x_i^2$ divides $\In_{\sigma}(f)$, so $\sigma \notin \Delta_{\R {\rm ad}}(I)$. Otherwise, let $d = \gcd(|a_1-b_1|, \hdots, |a_n-b_n|)$, meaning that there are $d+1$ lattice points on the edge joining $a$ and $b$. Using $v = (a-b)/d$, for some $\gamma_k \in \R$  
we can write \[\In_{\sigma}(f) \;\;=\;\; \sum_{k=0}^d \gamma_k \underline{x}^{(b+k v)} \;\;=\;\; \underline{x}^{b}\sum_{k=0}^d \gamma_k (\underline{x}^v)^k.\]
Then $\sigma \in \Delta_{\R {\rm ad}}(I)$ if and only if the polynomial in one variable $\sum_{k=0}^d \gamma_k t^k$ is real radical, meaning that all of its roots are real and distinct. This can easily be checked using Sturm sequences \cite[\S2.2.2]{sturm}.  Similarly, we can check if a cone of codimension $d$ belongs to $\Delta_{\R {\rm ad}}(I)$ by checking whether or not a certain polynomial in $d$ variables is real radical, though this is harder when $d\neq 0,1$.  

If the ideal $I$ is binomial, then for all $w \in \Trop(I)$, we have $\In_w(I) = I$. Thus to understand $\Trop_{\R {\rm ad}}(I)$ it suffices to know whether or not $I$ is real radical. Becker et. al. \cite{binom rad} present a concrete algorithm for computing the real radical of a binomial ideal.
If $I$ is not real radical,  $\Trop_{\R {\rm ad}}(I) = \emptyset$. If $I$ is real radical, then $\Trop_{\R {\rm ad}}(I)=\LL(\V_{\R^*}(I)) =\Trop_{\R^*}(I) = \Trop(I)$. If in addition $(\R_+)^n \cap \Trop(I)$ is nonempty, then by Theorem \ref{stable} the preorder $\sum \R[\underline{x}]^2 +I$ is stable.

When $I$ is neither principal nor binomial, we can use more general heuristics for computing $\Delta_{\R {\rm ad}}(I)$, which can be specialized to $\Trop_{\R {\rm ad}}(I)$ or $|\Delta_{\R {\rm ad}}(I)| \cap (\R_+)^n$.

\begin{alg} Given an ideal $I \subset \R[\underline{x}]$, calculate the fan $\Delta_{\R {\rm ad}}(I)$ as follows:
\begin{enumerate}
 \item Calculate $\overline{I}$, the homogenization of $I$. 
 \item Use {\tt GFan} to calculate the Gr\"obner fan of $\overline{I}$, $\Delta_{Gr}(\overline{I})$.
 \item Intersect the cones of $\Delta_{Gr}(\overline{I})$ with $\{w_0=0\}$ to obtain $\Delta_{Gr}(I)$. 
\item For $i = n, n-1, \hdots, 0$ and cones $\sigma \in \Delta_{Gr}(I)$ of dimension $i$, check if $\In_{\sigma}(I)$ is real radical.
	\subitem if yes, $\sigma \in \Delta_{\R {\rm ad}}(I)$ and for all faces $\tau$ of $\sigma$, $\tau \in \Delta_{\R {\rm ad}}(I)$. 
	\subitem if no, $\;\sigma \notin \Delta_{\R {\rm ad}}(I)$.
\end{enumerate}\end{alg}

Full dimensional cones $\sigma \in \Delta_{Gr}(I)$ ($i=n$ in step 4) correspond to monomial initial ideals. A monomial ideal is real radical if and only if it is radical if and only if it is square free. More generally, one can calculate the radical of an ideal with Gr\"obner basis methods and exclude $\sigma$ from $\Delta_{\R {\rm ad}}(I)$ whenever $\In_{\sigma}(I)$ is not radical.  In general as the dimension of $\sigma \in \Delta_{Gr}(I)$ decreases, $\In_{\sigma}(I)$ has less structure and checking whether $\In_{\sigma}(I)$ is real radical becomes more difficult. 

%To satisfy the conditions of Theorem \ref{stable} we only need to find one vector $w\in (\R_+)^n$ for which $\In_w(I)$ is real radical. If we are lucky, this will happen for $w$ in a high dimensional cone $\sigma \in \Delta_{Gr}(I)$. 

Currently the only general methods of determining whether or not an ideal is real radical are not practical for computation.  
Becker and Neuhaus present an algorithm for computing the real radical of an ideal via quantifier elimination \cite{comp1, comp2}.
For $I = \langle f_1, \hdots, f_r \rangle \subseteq \R[x_1, \hdots, x_n]$ 
they show that the real radical of $I$, $\sqrt[\R]{I}$ is generated by polynomials of degree at most 
\[ \max_{i=1, \hdots, r} \{\deg(f_i) \} ^{2^{O(n^2)}}.\] 
This also provides a bound for the computation time of their algorithm for finding $\sqrt[\R]{I}$. 
The field of computational real algebraic geometry and semidefinite programming is progressing quickly, so it may soon be possible to efficiently check if an arbitrary ideal is real radical. \\

Corollary \ref{trops} also has consequences for the compactness of a real variety. Consider an ideal $I \subset \R[\underline{x}]$. By the definition of logarithmic limit sets, we have that $\V_{\R^*}(I)$ is compact and nonempty if and only if its logarithmic limit set is the origin,  $\LL(\V_{\R^*}(I)) =\{0\}$, and $\V_{\R^*}(I)$ is empty if and only if $\LL(\V_{\R^*}(I))$ is empty.  Corollary \ref{trops} shows that $\Trop_{\R {\rm ad}}(I)$ and $\Trop_{\R^*}(I)$ provide inner and outer approximations of $\LL(\V_{\R^*}(I))$. This gives the following:
\begin{cor} \label{cor:compact} For an ideal $I\subset \R[\underline{x}]$,
\begin{enumerate}
\item[{\rm (a)}] if $\Trop_{\R^*}(I) \subseteq \{0\}$, then $\V_{\R^*}(I)$ is compact,
\item[{\rm (b)}] if $\Trop_{\R {\rm ad}}(I) \not\subseteq \{0\}$, then $\V_{\R^*}(I)$ is not compact, and
\item[{\rm (c)}] if $\Trop_{\R {\rm ad}}(I)\not\subseteq (\R_{\leq0})^n$, then $\V_{\R}(I)$ is not compact. 
\end{enumerate}
\end{cor}

This provides a method of verifying the compactness (or non-compactness) of the real variety of an ideal based only on its initial ideals in some cases. However, the example on page \pageref{dissonance} shows these conditions cannot completely characterize compactness. 

\begin{example} Let's see Corollary~\ref{cor:compact} in action: \smallskip\\
(a) Let $I=\langle (x-2)^2+(y-2)^2-1\rangle$. For every vector $w \neq \underline{0}$, the set $\V_{\R^*}(\In_w(I))$ is empty. For instance, $\In_{(0,-1)}(I) =\langle (x-2)^2 +3 \rangle$. Thus $\Trop_{\R^*}(I) = \{0\}$, confirming that $\V_{\R^*}(I)$ is compact.  \medskip\\
(b) Consider $I = \langle x^2+y^2-1\rangle$. One checks that $\Trop_{\R {\rm ad}}(I)$ is the union of the non-positive $x$ and $y$ axes, which shows that $\V_{\R^*}(I)$ is not compact (even though $\V_{\R}(I)$ is). \medskip \\
(c) For the ideal $I = \langle x^4+x^2y^2-1\rangle$, we see that $\In_{(-1,1)}(I) = \langle x^2y^2-1\rangle=\langle (xy+1)(xy-1)\rangle$ and thus $(-1,1)\in\Trop_{\R {\rm ad}}(I)$. This shows that the curve $\V_{\R}(I)$ is not compact.
\end{example}
As $\Trop_{\R^*}(I)$ and $\Trop_{\R{\rm ad}}(I)$ are imperfect approximations to $\LL(\V_{\R^*}(I))$, none of the converses of Corollary~\ref{cor:compact} hold.  Example~\ref{ex:Vbar} shows that the converse of (a) does not hold, and the example on page~\pageref{dissonance} provides a counterexample to the converses of (b) and (c).  

Lemma~\ref{R^n} provides a slightly more general condition for the non-compactness of $\V_{\R}(I)$. As discussed in Remark~\ref{rem:primDecop}, we can use $\{w\in \Trop(I): \In_w(I) \text{ has a real radical primary component}\}$ in place of 
$\Trop_{\R {\rm ad}}(I)$ in Corollary~\ref{cor:compact}.

\begin{example}\label{ex:primedec}
For any $f\in \R[x,y]$ with $\deg_{(2,1)}(f) < 10$, consider $I = \langle x^5 - x^4 y^2 + x^3 y^4 - x^2 y^6+f\rangle$.  Then
\[ \In_{(2,1)}(I) =  \langle x\rangle^2 \cap \langle x-y^2 \rangle \cap \langle x^2+y^4\rangle.\] 
As $\langle x-y^2\rangle$ is real radical, we see that there is a nonsingular point in $\V_{\R^*}(\In_{(2,1)}(I))$, for example the point $(1,1)$. Thus by Lemma~\ref{R^n}, the vector $(2,1)$ lies in $\LL(\V_{\R^*}(I))$ and $\V_{\R}(I)$ is not compact. 
\end{example}

\section*{Acknowledgements} Many thanks to Bernd Sturmfels for his useful comments and guidance, to Morgan Brown for many fruitful discussions, and to Claus Scheiderer for his suggestions and insights. The author was funded by NSF grant DMS-0757236 (FRG: Semidefinite Optimization and Convex Algebraic Geometry) and the University of California - Berkeley Mentored Research Award.


\begin{thebibliography}{BGN}\setlength{\itemsep}{-.5mm}

\bibitem[A]{log limit} D. Alessandrini, \textit{Logarithmic limit sets of real semi-algebraic sets}. arXiv:0707.0845v2 
\bibitem[BPR]{sturm} S. Basu, R. Pollack, M.F. Roy, \textit{Algorithms in real algebraic geometry}. Second edition. Algorithms and Computation in Mathematics, 10. Springer-Verlag, Berlin, 2006.
\bibitem[BGN]{binom rad} E. Becker, R. Grobe, M. Niermann, \textit{Radicals of binomial ideals. Algorithms for algebra.}  J.~Pure Appl. Algebra  117/118  (1997), 41--79.
\bibitem[BN]{comp1}E. Becker, R. Neuhaus,
\textit{Computation of real radicals of polynomial ideals.} Computational algebraic geometry, Progr. Math., 109, BirkhŠuser Boston, Boston, MA, (1993) 1--20.
\bibitem[BZ]{minors} D. Bernstein, A. Zelevinsky, \textit{Combinatorics of maximal minors.} J. Algebraic Combin. 2 (1993), no. 2, 111--121.
\bibitem[ET]{ET} M. Einsiedler, S. Tuncel. \textit{When does a polynomial ideal contain a positive
polynomial?} J.~Pure~Appl.~Algebra 164 (2001), 149Ð152.
\bibitem[E]{comm alg} D. Eisenbud, \textit{Commutative algebra. With a view toward algebraic geometry.} Graduate Texts in Mathematics, 150. Springer-Verlag, New York, 1995.
\bibitem[H]{Hart} R. Hartshorne,  \textit{Algebraic geometry.} Graduate Texts in Mathematics, No. 52. Springer-Verlag, New York-Heidelberg, 1977.
\bibitem[J]{gfan} A. Jensen, \textit{Gfan, a software system for Gr\"obner fans and tropical varieties}. Available at http://www.math.tu-berlin.de/$\sim$jensen/software/gfan/gfan.html.

\bibitem[K]{P-M} W. Kulpa, \textit{The Poincar\'e-Miranda theorem.}  Amer. Math. Monthly  104  (1997),  no. 6, 545--550. 
\bibitem[MS]{trop book} D. Maclagan, B. Sturmfels, \textit{Introduction to Tropical Geometry.} (book manuscript) http://math.berkeley.edu/$\sim$bernd/math274.html
\bibitem[M]{PPSS} M. Marshall, \textit{Positive polynomials and sums of squares.} Mathematical Surveys and Monographs, 146. American Mathematical Society, Providence, RI, 2008.
\bibitem[N]{Netzer} T. Netzer, \textit{Stability of quadratic modules.} Manuscripta Math.  129  (2009),  no. 2, 251--271.
\bibitem[Ne]{comp2} R. Neuhaus, \textit{Computation of real radicals of polynomial ideals. II.} J. Pure Appl. Algebra 124 (1998), no. 1-3, 261--280. 
\bibitem[PS]{SP} V. Powers, C. Scheiderer, \textit{The moment problem for non-compact semialgebraic sets.} Adv. Geom. 1  (2001),  no. 1, 71--88.
\bibitem[RST]{trop2}J. Richter-Gebert, B. Sturmfels, T. Theobald, \textit{First steps in tropical geometry}.  Idempotent mathematics and mathematical physics, Contemp. Math., 377 (2005), 289--317,
\bibitem[S]{Sch}C. Scheiderer, \textit{Non-existence of degree bounds for weighted sums of squares representations,} J. Complexity 21 (2005), 823-844.
\bibitem[Sch]{Schmudgen} K. Schm\"udgen, \textit{The K-moment problem for compact semi-algebraic sets.} Math Ann. 289 (1991), 203--206.
\bibitem[SW]{real trop} D. Speyer, L. Williams, \textit{The tropical totally positive Grassmannian.} J. Algebraic Combin. 22 (2005), no. 2, 189--210.
\bibitem[St]{UGB} B. Sturmfels, \textit{Gr\"obner bases and convex polytopes.} University Lecture Series, 8. American Mathematical Society, Providence, RI, (1996). 
\bibitem[V]{viro} O. Viro, \textit{ From the sixteenth Hilbert problem to tropical geometry.} Japan. J. Math. 3 (2008), no. 2, 185--214. 

\end{thebibliography}
\end{document}